\documentclass[11pt]{article}
\usepackage[margin=1in]{geometry}
\usepackage{amsmath,amssymb,amsfonts,amsthm,bbm}
\usepackage{mathtools}
\usepackage{graphicx}
\graphicspath{{./},{./figures},{./figure_p2},{./figure_p5}}
\usepackage{tikz}
\usepackage{xcolor}
\usepackage{algorithm}
\usepackage{algpseudocode}
\usepackage{hyperref}
\hypersetup{colorlinks=true, linkcolor=blue, citecolor=blue, urlcolor=blue}
\usepackage{authblk}
\usepackage{microtype}

\newcommand{\calG}{{\mathcal G}}
\newcommand{\calV}{{\mathcal V}}
\newcommand{\calE}{{\mathcal E}}

\newcommand{\R}{\mathbb{R}}
\newcommand{\Prob}{\operatorname{Prob}}

\newcommand{\dd}{\mathrm{d}}

\newcommand{\argmin}{\mathop{\rm argmin}}

\newcommand{\blue}{\color{blue}}

\newtheorem{theorem}{Theorem}

\newtheorem{proposition}[theorem]{Proposition}
\newtheorem{problem}{Problem}
\newtheorem{definition}{Definition}

\newtheorem{remark}{Remark}

\newtheorem{assumption}{Assumption}

\title{Optimal transport control of network flow under Fundamental Diagram constraints}

\author[1,2]{Anqi Dong\thanks{\texttt{anqid@kth.se}}}
\author[1,3]{Karl Henrik Johansson\thanks{\texttt{kallej@kth.se}}}
\author[2,3]{Johan Karlsson\thanks{\texttt{johan.karlsson@math.kth.se}}}

\affil[1]{Division of Decision and Control Systems, KTH Royal Institute of Technology, SE-100 44 Stockholm, Sweden}
\affil[2]{Department of Mathematics, KTH Royal Institute of Technology, SE-100 44 Stockholm, Sweden}
\affil[3]{Digital Futures, KTH Royal Institute of Technology, SE-100 44 Stockholm, Sweden}

\date{} 

\begin{document}
\maketitle

\begingroup\renewcommand\thefootnote{}\footnote{This research is supported by Swedish Research Council Distinguished Professor Grant 2017-01078, Knut and Alice Wallenberg Foundation Wallenberg Scholar Grant, Swedish Research Council (VR) under grant 2020-03454, and KTH Digital Futures.}\addtocounter{footnote}{-1}\endgroup

\begin{abstract}
Optimal transport (OT) theory provides a principled framework for modeling mass movement in applications such as mobility, logistics, and economics. Classical formulations, however, generally ignore capacity limits that are intrinsic in applications, in particular in traffic flow problems. We address this limitation by incorporating fundamental diagrams into a dynamic continuous-flow OT model on graphs, thereby including empirical relations between local density and maximal flux.  We adopt an Eulerian kinetic action on graphs that preserves displacement interpolation in direct analogy with the continuous theory. Momentum lives on edges and density on nodes, mirroring road-network semantics in which segments carry speed and intersections store mass. The resulting fundamental-diagram-constrained OT problem preserves mass conservation and admits a convex variational discretization, yielding optimal congestion-aware traffic flow over road networks. We establish existence and uniqueness of the optimal flow with sources and sinks, and develop an efficient convex optimization method. Numerical studies begin with a single-lane line network and scale to a city-level road network.%
\footnote{The code to conduct all experiments can be found at {\blue \url{https://github.com/dytroshut/continuous_net_flow}}}

\end{abstract}

\begin{keywords}
Optimal transport, fundamental diagram, traffic flow control, flow constraints.
\end{keywords}

\section{Introduction}\label{sec:intro}

The fundamental diagram \cite{greenshields1935study,siebel2006fundamental} is an empirical relation between density and admissible flow that captures the capacity limits of macroscopic motion on road networks. Since Greenshields’ measurements \cite{greenshields1935study}, it has served as a foundation for conservation law models of vehicular traffic and pedestrian crowds. The Lighthill–Whitham–Richards (LWR) model \cite{lighthill1955kinematic1,richards1956shock} embeds this relation in a first-order conservation law and explains shock formation, queue spillback, and capacity drop, with discrete solvers such as the cell transmission model (CTM) \cite{daganzo1994cell}. Extensions to two dimensions cover lane-free motion and open domains. However, these PDE based descriptions are largely descriptive and are typically used in open loop. Direct use for large scale planning and control is difficult because of nonlinearities, boundary and interface conditions, and calibration across heterogeneous regions. This motivates a formulation that retains conservation and fundamental diagram constraints while enabling tractable optimization and feedback design.

Optimal transport \cite{villani2021topics,rachev1998mass} offers a variational viewpoint. In the formulation of Benamou and Brenier, transport is posed as minimum energy steering of the continuity equation \cite{benamou2000computational}. When mass is supported on a graph, replacing spatial derivatives with the incidence operator yields a continuous flow distance that is computationally scalable and compatible with efficient solvers \cite{solomon2016continuous, treleaven2013explicit,maas2011gradient,chen2016robust}. Classical OT, however, does not account for congestion. Without capacity limits, flows through narrow or saturated links can become unrealistically large under high demand. Recent efforts address this limitation by imposing 
momentum bounds \cite{dong2024monge,stephanovitch2024optimal,haasler2024scalable} or by showing that continuous flow can be constrained to abide by a beta family of fundamental diagrams \cite{dong2025fundamental}. These directions are promising, yet high-dimensional representations and large network instances remain challenging for planning and control at scale. This motivates integrating fundamental diagram constraints directly into a graph-based countinuous-flow dynamic OT model that preserves conservation and admits tractable optimization. 

Motivated by these observations, we propose a salient framework embedding fundamental diagram constraints into a graph-based dynamic OT model. Densities are assigned to vertices and link flows to directed edges, with conservation enforced by a discrete continuity equation. On each edge, admissible flow is bounded by a concave fundamental diagram (e.g., Greenshields) evaluated at a midpoint density. A midpoint time discretization then produces a convex program whose action is strictly convex in the flows, which implies a unique minimizer under standard convex constraints. The resulting trajectories inherit the smooth displacement behavior of OT while respecting capacity, and the formulation exposes linear conservation structure and per-edge convex capacity sets that are amenable to scalable optimization and feedback design. To handle city-scale networks and long horizons, we require a method that exploits sparse incidence structure and per-edge separability, enabling parallel capacity projections and predictable memory. We derive an augmented-Lagrangian splitting method \cite{boyd2011distributed, hestenes1969multiplier, bertsekas2014constrained} in which the flow update is a sparse SPD solve, the capacity step is an elementwise projection onto the fundamental-diagram hypograph, and the density update is a strictly convex problem with block-tridiagonal structure. We demonstrate the approach on a single-lane line network and on a planar subgraph of the Athens road network, where it reproduces congestion patterns and converges to CVX \cite{grant2008cvx} baselines.

The rest of the paper is organized as follows. Section~\ref{sec:prelim} reviews the dynamic OT formulation and the fundamental diagram/LWR model as motivation. Section~\ref{sec:graph-trans} formulates FD-OT on graphs and establishes basic well-posedness. Section~\ref{sec:optimization} presents an ADMM-based solver with efficient updates. Section~\ref{sec:numerical} reports numerical studies, and Section~\ref{sec:conclusion} concludes with remarks and future directions.

\section{Preliminaries}\label{sec:prelim}

\subsection{Dynamical optimal mass transport}

The Benamou--Brenier formulation of OT \cite{benamou2000computational} describes transportation with respect to the instantaneous kinetic energy
\begin{equation*}
\mathcal{L}(\rho,m) = \int_{\mathbb{R}^n} \frac{\|m(x)\|^2}{2\,\rho(x)} \, dx
\end{equation*}
between two probability densities $\rho_0,\rho_1 \in \mathcal{P}(\mathbb{R}^n)$ over a fixed transportation interval $[0,1]$. The squared 2--Wasserstein distance is then the minimal action
\begin{align*}
\mathcal W_2^2(\rho_0,\rho_1) = \min_{(\rho,m)} \int_0^1 \mathcal{L}\big(\rho(t),m(t)\big) \, dt,
\end{align*}
subject to the continuity equation
\begin{equation}\label{eq:conserve}
\partial_t \rho(t,x) + \nabla_x \cdot m(t,x) = 0,
\end{equation}
with boundary conditions $\rho(0)=\rho_0$ and $\rho(1)=\rho_1$, and momentum field $m:=\rho v$ with $v$ the velocity field.

From a control perspective, this is the problem of steering the continuity equation with minimal energy \cite{chen2016optimal}. The optimal solution defines a geodesic in Wasserstein space, and the corresponding density path is uniquely determined. At the particle level, each trajectory follows constant velocity, and the density evolves through displacement interpolation,
\begin{equation*}
\rho_t = \big((1-t)\mathrm{Id} + tT\big)_{\#}\rho_0, \quad t \in [0,1],
\end{equation*}
where $T$ is the optimal Monge map from $\rho_0$ to $\rho_1$ \cite{gangbo1996geometry,villani2021topics}.

\subsection{Fundamental diagram and LWR model}

The fundamental diagram captures the empirical observation that flow capacity on a link depends on its density. It represents the intrinsic congestion constraint of physical networks and provides the basis for many macroscopic traffic flow models. We first recall the fundamental diagram and then its incorporation into the LWR model before presenting our relaxed optimal transport formulation.

\begin{definition}[\bf Fundamental diagram]
The fundamental diagram specifies the relation between density $\rho$ and admissible momentum $m$. In its classical Greenshields form, the momentum $m(\rho)$ is given by
\begin{equation}\label{eq:fd}
\mathcal{Q}(\rho) = v_0 \, \rho \Bigl(1 - \rho/\hat\rho\Bigr),
\qquad 0 \le \rho \le \hat\rho,
\end{equation}
where $v_0$ is the free-flow speed and $\hat\rho>0$ is the jam density. The function $\mathcal{Q}(\rho)$ is concave, vanishes at $\rho=0$ and $\rho=\hat\rho$. The velocity-density and momentum-density relations are illustrated in Figure~\ref{fig:fundamental-diagram}.
\end{definition}

\begin{figure}[htb!]
    \centering
    \includegraphics[width=0.4\linewidth]{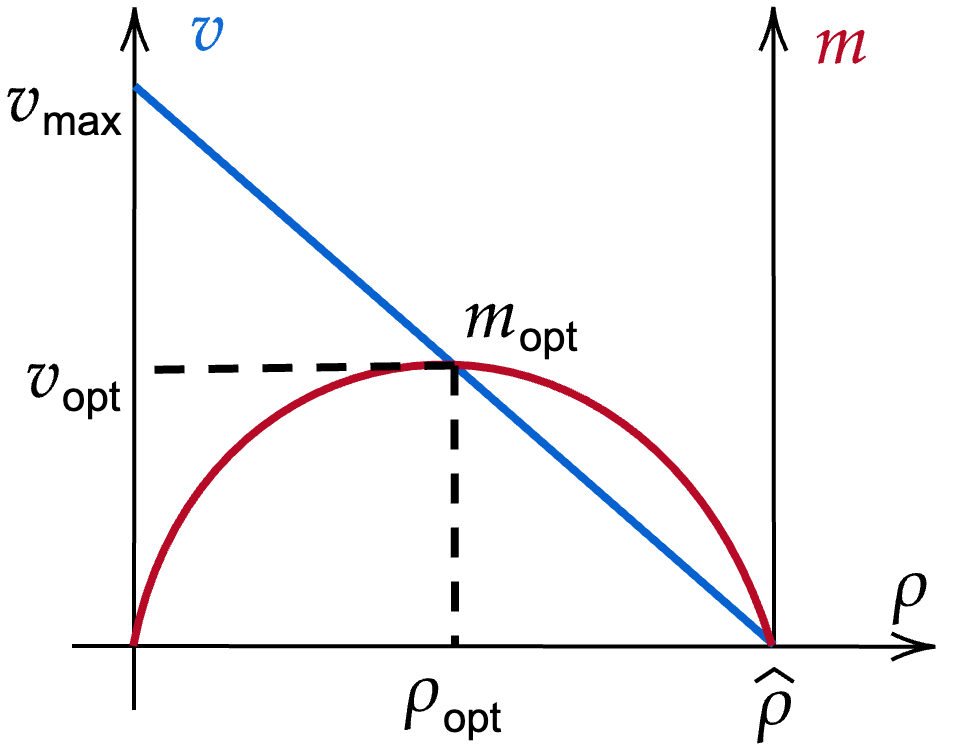}
    \caption{Greenshields fundamental diagram. The blue line shows the velocity-density relation, while the red curve shows the momentum (flux) as a function of density. The optimal operating point $(\rho_{\text{opt}}, m_{\text{opt}})$ corresponds to the maximum flux, attained at half the jam density $\hat\rho/2$ with velocity $v_{\text{opt}}$.}
    \label{fig:fundamental-diagram}
    \vspace{-1em}
\end{figure}

\begin{definition}[\bf LWR model]
The Lighthill–Whitham–Richards (LWR) model incorporates the fundamental diagram \eqref{eq:fd} into the conservation law \eqref{eq:conserve}, i.e., 
\begin{equation*}
\partial_t \rho(t,x) + \partial_x \mathcal{Q}\bigl(\rho(t,x)\bigr) = 0,
\end{equation*}
where $\rho(t,x)$ denotes the density field. This model predicts shock waves, spillback, and capacity drop, and is widely used for macroscopic traffic flow analysis.
\end{definition}

\subsection{Fundamental-diagram-constraint OT}
While the LWR model prescribes flux as a nonlinear function of density, we take a different viewpoint motivated by optimal transport. We keep the continuity equation in its general form $\partial_t \rho(t,x) + \nabla \cdot m(t,x) = 0$, and regard the momentum field $m$ as a decision variable, which can be controlled, e.g., by imposing variable speed limits. Therefore, the fundamental diagram then appears as a convex constraint
\begin{subequations}
\begin{equation*}
0 \le m(t,x) \le \mathcal{Q}\bigl(\rho(t,x)\bigr),
\end{equation*}
and we seek to minimize the kinetic energy action
\begin{equation*}
\int_0^1 \int_{\R^n} \frac{\|m(t,x)\|^2}{2\,\rho(t,x)} \,\dd x \,\dd t,
\end{equation*}    
\end{subequations}
subject to the continuity equation and the capacity constraint. 

This relaxation connects the classical LWR model, which specifies flux as a function of density, with the dynamic viewpoint of optimal transport. LWR was originally one-dimensional, but has been extended to two-dimensional versions and used for lane-free traffic, urban grids, and pedestrian dynamics, capturing anisotropic propagation of congestion waves and complex shock interactions
\cite{mollier2019two,agrawal2023two}. These PDE-based models are difficult to analyze and to simulate at scale, and their nonlinear structure is not directly amenable to optimization and feedback design.

A natural step for control is to discretize space into cells and pose the dynamics on a network: densities live on nodes, flows on edges, and the fundamental diagram bounds admissible transfers on each edge. This yields a relaxed, controllable formulation that preserves conservation and congestion while exposing linear conservation constraints and convex per-edge capacity sets and justifies a network-based formulation of congestion-aware transport that is relevant to large-scale traffic management, evacuation planning in pedestrian systems.

\section{Graph-based optimal transportation}\label{sec:graph-trans}

We follow the continuous flow formulation as in~\cite{solomon2016continuous}. The domain is a finite connected network $\calG=(\calV,\calE)$ with vertex set $\calV=\{v_1,\dots,v_n\}$ and directed edge set $\calE=\{e_1,\dots,e_m\}$ obtained by duplicating each undirected edge with both orientations. For $e_\ell=(v_i\to v_j)\in\calE$, $v_i$ is the tail and $v_j$ is the head. The incidence operator $D\in\R^{m\times n}$ is
\[
D_{\ell k}=\begin{cases}
-1,& v_k=v_i,\\[2pt]
+1,& v_k=v_j,\\[2pt]
0,& \text{otherwise},
\end{cases}
\]
so that $(Dx)_\ell = x(v_j)-x(v_i)$, for all $x\in\R^n$. The vertex divergence of an edge flow $m\in\R^m$ is $D^\top m$, i.e.,
\[
(D^\top m)(v_k)=\!\!\sum_{e_\ell:\,v_k=\text{head}(e_\ell)}\!\! m_\ell\;-\!\!\sum_{e_\ell:\,v_k=\text{tail}(e_\ell)}\!\! m_\ell,
\]
which represents the inflow of flux at $v_k$ -- the first sum collects flow entering $v_k$, the second flow leaving. This is the graph analogue of the spatial divergence $\nabla_x\cdot m$ and yields the discrete continuity law $\dot\rho + D^\top m = 0$. The space of positive probability measures on vertices is
\begin{equation*}
\Prob(\calG)=\Bigl\{\rho\in\R^{n}_{\geq0}\ \Bigm|\ \sum_{v\in\calV}\rho(v)=1\Bigr\}.
\end{equation*}
Following the continuous flow formulation in \cite{solomon2016continuous}, we consider trajectories $\rho(t)\in\Prob(\calG)$ and nonnegative edge momentum $m(t)\in\R_{\ge0}^{m}$ satisfying the discrete continuity equation
\[
\dot\rho(t)=D^\top m(t),\ \rho(0)=\rho_0, \mbox{ and } \rho(1)=\rho_1.
\]
The kinetic energy functional on network is defined as 
\begin{equation*}
\mathbf{L}(\rho,m) \;=\!\!\!\! \sum_{e_\ell=(v_i\to v_j)} \frac{m(t,e_\ell)^2}{2}\!\left(\frac{1}{\rho(t,v_i)}+\frac{1}{\rho(t,v_j)}\right).
\end{equation*}
The squared continuous flow transport distance between $\rho_0$ and $\rho_1$ is the minimal action
\begin{equation*}
\mathcal{W}_2^2(\rho_0,\rho_1)\;=\; \min_{(\rho,m)} \int_0^1 \mathbf{L}(\rho(t),m(t))\,dt,
\end{equation*}
where the optimum is over all admissible $(\rho,m)$ satisfying the continuity equation. This construction yields a metric on $\Prob(\calG)$ and induces a Riemannian structure under which minimizers are constant speed geodesics \cite{solomon2016continuous}.

For numerical purposes, we use a midpoint time discretization. Densities are sampled at $t_i=i/k$ as $\rho_0,\ldots,\rho_k$, and edge momentum at midpoints $\tau_i=(i-\tfrac{1}{2})/k$ as $m_1,\ldots,m_k$. For an oriented edge $e_\ell=(v_a\to v_b)$, define
\begin{equation*}
\bar{\rho}_{e_\ell,i} \;=\; \frac{\rho_{i-1}(v_a) + \rho_i(v_b)}{2}, \quad e_\ell\in\mathcal{E},\ i=1,\ldots,k.
\end{equation*}
The discrete action is
\begin{equation*}
\mathcal{J}(\rho,m) \;=\; \sum_{i=1}^{k} \sum_{e_\ell\in\mathcal{E}} \frac{k\, m_i(e_\ell)^2}{2\,\bar{\rho}_{e_\ell,i}},
\end{equation*}
subject to the discrete continuity relations
\begin{equation*}
\rho_i - \rho_{i-1} \;=\; D^\top m_i, \qquad i=1,\ldots,k,
\end{equation*}
with endpoints $\rho_0$ and $\rho_k$ prescribed. This yields a convex action with linear conservation constraints; algorithmic properties and complexity are discussed later.

To model congestion and flow limitations on networks such as transportation or communication systems, we impose a fundamental diagram that constrains edge momentum as a function of local density. We adopt the classical Greenshields relation.

\begin{definition}[\bf Edge-wise Fundamental Diagram]\label{def:greenshields}
For each directed edge $e_\ell=(v_a\to v_b)\in\mathcal{E}$, the maximum allowable momentum is given by
\begin{equation*}
\mathcal{Q}_{e_\ell}(\rho) \;=\; v_{0}^{e_\ell}\,\rho\!\left(1 - \frac{\rho}{\hat\rho^{e_\ell}}\right),
\qquad 0 \le \rho \le \hat\rho^{e_\ell},
\end{equation*}
where $v_{0}^{e_\ell}>0$ is the free flow speed and $\hat\rho^{e_\ell}>0$ is the jam density. At time $t$, the midpoint density on $e_\ell=(v_a\to v_b)$ is
\begin{equation}\label{eq:mid-point-rho}
\bar\rho_{e_\ell}(t) \;=\; \tfrac{1}{2}\big(\rho(t,v_a)+\rho(t,v_b)\big),
\end{equation}
and the feasible momentum field $m(t)\in\mathbb{R}_{\ge 0}^{m}$ must satisfy the capacity constraint, i.e.,
\begin{equation*}
0 \le m(t,e_\ell) \;\le\; \mathcal{Q}_{e_\ell}\big(\bar\rho_{e_\ell}(t)\big), \;\;\forall\, e_\ell\in\mathcal{E},\forall\, t\in[0,1].
\end{equation*}
\end{definition}
The velocity to density and momentum to density relations are illustrated in Fig.~\ref{fig:fundamental-diagram}. The problem thus read as follows.

\begin{problem}[\bf Continuous-flow over traffic network]\label{prob:1}
Given initial and terminal densities $\rho_0, \rho_1 \in \Prob(\calG)$, find a pair $(\rho(t), m(t))$ that minimizes the action functional
\begin{subequations}
\begin{align}\label{eq:objective}
\int_0^1 \sum_{e_\ell=(v_i\to v_j)} \frac{m(t,e_\ell)^2}{2} \left( \frac{1}{\rho(t,v_i)} + \frac{1}{\rho(t,v_j)} \right)\, dt,
\end{align}    
subject to the continuity constraints
\begin{equation*}
\dot \rho(t) = D^\top m(t)    
\end{equation*}
with the marginal constraint $\rho(0) = \rho_0$, $\rho(1) = \rho_1$ and the fundamental-diagram constraint for congestion, i.e.,
\begin{equation*}
0 \le m(t,e_\ell) \le \mathcal{Q}_{e_\ell} \bigl( \bar\rho_{e_\ell}(t) \bigr),\ \ \forall\, e_\ell \in \calE. 
\end{equation*}
\end{subequations}
\end{problem}

We are now ready to establish convexity and uniqueness, under appropriate existence conditions. Note that the objective \eqref{eq:objective} is convex but not strictly convex. 
\begin{assumption}[\bf Existence]\label{ass:existence} There exists at least one admissible solution $(\rho,m)$ 
with $\rho:[0,1]\to\Prob(\calG)$ and $m\in L^2\big([0,1],\R_{\ge0}^{m}\big)$ satisfying the continuity equation and the boundary conditions, and the fundamental–diagram capacity constraints $0\le m(t,e_\ell)\le \mathcal{Q}_{e_\ell}\!\big(\bar\rho_{e_\ell}(t)\big)$ $\forall t\in[0,1]$ and $e_\ell\in\calE$.
\end{assumption}

\begin{proposition}[\bf Uniqueness]\label{prop:unique}
Under Assumption~\ref{ass:existence}, Problem~\ref{prob:1}
admits a unique minimizer $(\rho^\star,m^\star)$. 
\end{proposition}

\begin{proof}
Assume there exist two distinct optimal solutions $(\rho^{(1)},m^{(1)})$ and $(\rho^{(2)},m^{(2)})$.
For any $\theta\in(0,1)$ define
\begin{align*}\label{eq:theta-def}
\rho^{\theta}&:=\theta\,\rho^{(1)}+(1-\theta)\,\rho^{(2)},\\
m^{\theta}&:=\theta\,m^{(1)}+(1-\theta)\,m^{(2)}.    
\end{align*}
Feasibility is preserved. The continuity equation and boundary data are affine, so that $\dot\rho^{\theta}=D^\top m^{\theta}$, $\rho^{\theta}(0)=\rho_0$, and $\rho^{\theta}(1)=\rho_1$. The capacity constraint is also satisfied, i.e., for any time $t$ and edge $e=(v_i\to v_j)$, we have 
\begin{align*}
\bar\rho^{\theta}_{e}(t)
= \frac12\big(\rho^{\theta}(t,v_i)+\rho^{\theta}(t,v_j)\big) 
= \theta\,\bar\rho^{(1)}_{e}(t)+(1-\theta)\,\bar\rho^{(2)}_{e}(t),
\end{align*}
and since $m^{(k)}(t,e)\in[0,\mathcal Q_e(\bar\rho^{(k)}_e(t))]$ for $k=1,2$, we have 
\begin{align*}
m^{\theta}(t,e)&=\theta m^{(1)}(t,e)+(1-\theta)m^{(2)}(t,e)\\
&\le \theta\,\mathcal Q_e\!\big(\bar\rho^{(1)}_e(t)\big)
   +(1-\theta)\,\mathcal Q_e\!\big(\bar\rho^{(2)}_e(t)\big)
\le \mathcal Q_e\!\big(\bar\rho^{\theta}_e(t)\big),    
\end{align*}
as function $\mathcal Q_e$ is concave. 

For fixed $t$, the instantaneous energy is in the form
\[
\mathcal{L}(\rho(t),m(t))
=\sum_{e=(v_i\to v_j)} \frac{m(t,e)^2}{2}\Big(\frac{1}{\rho(t,v_i)}+\frac{1}{\rho(t,v_j)}\Big)
\]
is a sum of perspective terms $(m^2)/(2\rho)$, which are jointly convex in $(m,\rho)$ on $\{\rho>0\}$. Hence, for all $t\in[0,1]$, we have 
\begin{equation}\label{eq:conv-ineq}
\mathcal{L}\big(\rho^{\theta},m^{\theta}\big)
\le
\theta\mathcal{L}\big(\rho^{(1)},m^{(1)}\big)
+(1-\theta)\mathcal{L}\big(\rho^{(2)},m^{(2)}\big)
\end{equation}
and integrating over $t\in[0,1]$ yields to 
\begin{align*}
\int_0^1 \mathcal{L}\big(\rho^{\theta},m^{\theta}\big)\,dt
&\le
\theta\!\int_0^1 \mathcal{L}\big(\rho^{(1)},m^{(1)}\big)\,dt+(1-\theta)\!\int_0^1 \mathcal{L}\big(\rho^{(2)},m^{(2)}\big)\,dt.    
\end{align*}

Since equality in \eqref{eq:conv-ineq} holds only when the edgewise triples coincide almost everywhere, we have $m^{(1)}=m^{(2)}$ and $\rho^{(1)}=\rho^{(2)}$. With $m^\star:=m^{(1)}=m^{(2)}$ and $\rho(0)=\rho_0$, the linear relation $\dot\rho=D^\top m^\star$ has a unique solution, hence $\rho^{(1)}=\rho^{(2)}=:\rho^\star$ on $[0,1]$. This contradicts distinctness, and uniqueness follows.
\end{proof}

The single commodity formulation as in Problem~\ref{prob:1} extends directly to multiple commodities on the same network. Let $\mathcal{C}=\{1,\dots,K\}$ index the commodities sharing the same network $\calG=(\calV,\calE)$. For each $c\in\mathcal{C}$, denote by $\rho^{c}(t)\in\R_{\ge 0}^{n}$ the vertex density and by $m^{c}(t)\in\R_{\ge 0}^{m}$ the edge momentum; conservation holds component-wise via
\[
\dot\rho^{c}(t)=D^\top m^{c}(t),\qquad c\in\mathcal{C},
\]
so the total mass $\sum_{v\in\calV}\rho^{c}(t,v)=M^{c}$ is constant in time. Because commodities share physical capacity on each edge $e_\ell=(v_i\to v_j)$, their momentum add and are bounded by the fundamental diagram 
\begin{align*}
0 \le \sum_{c\in\mathcal{C}} m^{c}(t,e_\ell)\ \le\ \mathcal{Q}_{e_\ell}\!\big(\bar\rho^{\rm tot}_{e_\ell}(t)\big),    
\end{align*}
that evaluated at the total midpoint density, i.e., $\bar\rho^{\rm tot}_{e_\ell}(t):=\tfrac12\Big(\sum_{c}\rho^{c}(t,v_i)+\sum_{c}\rho^{c}(t,v_j)\Big)$. For a (possibly weighted) sum of kinetic actions across commodities, the feasible set remains convex and the integrand is strictly convex in $(m^{c})_{c\in\mathcal{C}}$. Consequently, optimal commodity momentum is unique almost everywhere in time, and by the continuity equations with fixed initial data, the corresponding trajectories are unique.

\begin{remark}[\bf Time varying fundamental diagram]
Edge capacities may vary in time to model incidents, weather, or control inputs such as speed limits. Let
\[
\mathcal{Q}_{e_\ell}(t,\rho)=v_0^{e_\ell}(t)\,\rho\Big(1-\rho/\hat\rho^{e_\ell}(t)\Big),\quad 0\le \rho\le \hat\rho^{e_\ell}(t),
\]
with $v_0^{e_\ell}(t)$ and $\hat\rho^{e_\ell}(t)$ bounded and measurable. The capacity constraint is
\[
0\le m(t,e_\ell)\ \le\ \mathcal{Q}_{e_\ell}\big(t,\bar\rho_{e_\ell}(t)\big),\quad e_\ell\in\calE.
\]
For each fixed $t$, the feasible set is convex and strict convexity in $m$ is unchanged. Under Assumption~\ref{ass:existence}, Proposition~\ref{prop:unique} applies without modification.
\end{remark}

\section{Optimization method}\label{sec:optimization}

We discretize the dynamic optimal transport problem on a directed graph
$\mathcal{G}=(\mathcal{V},\mathcal{E})$ with $n=|\mathcal{V}|$ nodes and
$m=|\mathcal{E}|$ directed edges. We discretize time into $k$ intervals and introduce $\boldsymbol{\rho}=\begin{bmatrix}\rho_1 & \dots & \rho_{k-1}\end{bmatrix}^\top
\in \mathbb{R}^{(k-1)n}$, $\mathbf{m}=\begin{bmatrix}m_1 & \dots & m_k\end{bmatrix}^\top
\in \mathbb{R}^{km}$, and $\mathbf{q}=\begin{bmatrix}q_1 & \dots & q_k\end{bmatrix}^\top
\in \mathbb{R}^{km}$, where $\rho_0,\rho_k\in\mathbb{R}^n$ are fixed initial and terminal densities,
$m_i\in\mathbb{R}^m$ are edge fluxes, and $q_i\in\mathbb{R}^m$ are slack variables for the
fundamental diagram (FD).

Recall $D\in\mathbb{R}^{m\times n}$ denotes the directed incidence matrix, we define the block operators $\mathbf{D}=\mathrm{blkdiag}(D^\top,\dots, D^\top)\in\mathbb{R}^{kn\times km}$, and 
\[
\Delta\boldsymbol{\rho}
=
\begin{bmatrix}
\rho_1 -\rho_0\\
\rho_2-\rho_1\\
\vdots\\
\rho_{k-1}-\rho_{k-2}\\
\rho_k-\rho_{k-1}
\end{bmatrix}
\in\mathbb{R}^{kn},
\]
so that the continuity constraints takes in the form $\mathbf{D}\,\mathbf{m}=\Delta\boldsymbol{\rho}$.

At each step $i$, the midpoint density along edge $e_\ell$ is $\bar\rho_{e_\ell,i}=\frac{1}{2}\big(\rho_{i-1}(v)+\rho_i(w)\big)$ as in~\eqref{eq:mid-point-rho}. The discrete kinetic energy is
\[
\mathcal{J}(\boldsymbol{\rho},\mathbf{m})
=\frac{1}{2}\sum_{i=1}^k m_i^\top W_i(\boldsymbol{\rho})\,m_i
=\frac{1}{2}\,\mathbf{m}^\top W(\boldsymbol{\rho})\,\mathbf{m},
\]
where $W(\boldsymbol{\rho})=\mathrm{blkdiag}(W_1,\dots,W_k)$ and, for $e_\ell=(v\to w)\in\mathcal{E}$, we have 
\[
\big[W_i(\boldsymbol{\rho})\big]_{e_\ell,e_\ell}
= k\!\left(\frac{1}{2\rho_{i-1}(v)}+\frac{1}{2\rho_i(w)}\right).
\]

We enforce edgewise capacity through the admissible FD set, evaluated at midpoint densities for $i=1,\dots,k$:
\begin{equation}\label{eq:fd-set}
\small \mathcal{K}_{\rm FD}(\boldsymbol{\rho})
=\Big\{\mathbf{q}\ \Big|\ 
0 \le q_i(e_\ell) \le \mathcal{Q}_{e_\ell}\big(\bar\rho_{e_\ell,i}\big)
\ \text{for all } e_\ell\in\mathcal{E}\Big\},
\end{equation}
where $\mathcal{Q}_{e_\ell}:[0,\hat\rho^{e_\ell}]\to\mathbb{R}_+$ is the concave flow–density relation on edge $e_\ell$
(e.g., Greenshields $\mathcal{Q}_{e_\ell}(\rho)=v_0^{e_\ell}\rho\,(1-\rho/\hat\rho^{e_\ell})$).
Equivalently, one may eliminate $\mathbf{q}$ and enforce
$0 \le m_i(e_\ell) \le \mathcal{Q}_{e_\ell}(\bar\rho_{e_\ell,i})$ for all $e_\ell$ and $i$.

\subsection{ADMM method}

We employ the alternating direction method of multipliers (ADMM) \cite{boyd2011distributed,bertsekas2014constrained,boyd2004convex}. The variable split is chosen to exploit the graph and time structure. We keep $\mathbf m$ as the momentum (one block per time step), $\boldsymbol{\rho}$ as node densities at intermediate snapshots (endpoints fixed), and introduce a consensus copy $\mathbf q$ that will carry the fundamental-diagram (FD) constraint by projection. The linear continuity relation couples consecutive snapshots, while the kinetic objective couples $\mathbf m$ to $\boldsymbol{\rho}$ through midpoint weights. With this split, we consider
\begin{align*}
\min_{\boldsymbol{\rho},\,\mathbf{m},\,\mathbf{q}\in\mathcal{K}_{\rm FD}(\boldsymbol{\rho})}\
& \frac12\,\mathbf{m}^\top W(\boldsymbol{\rho})\,\mathbf{m}\\
\text{subject to}\quad &
\mathbf{D}\mathbf{m}-\Delta\boldsymbol{\rho}=0,\\ &\mathbf{m}-\mathbf{q}=0.    
\end{align*}
By introducing the dual variable for the constraints, the augmented Lagrangian is
\begin{align*}
&\mathcal{L}(\boldsymbol{\rho},\mathbf{m},\mathbf{q};\lambda,\phi)
=\tfrac12\,\mathbf{m}^\top W(\boldsymbol{\rho})\,\mathbf{m}
+\lambda^\top(\mathbf{D}\mathbf{m}-\Delta\boldsymbol{\rho})\\
&+\tfrac{\beta}{2}\|\mathbf{D}\mathbf{m}-\Delta\boldsymbol{\rho}\|_2^2
+\phi^\top(\mathbf{m}-\mathbf{q})
+\tfrac{\gamma}{2}\|\mathbf{m}-\mathbf{q}\|_2^2 + \iota_{\mathcal K}(\mathbf{q}),    
\end{align*}
with $\iota(\cdot)$ as the indicator function.

An iteration of ADMM is to update $\mathbf m, \mathbf q,$ and $\boldsymbol{\rho}$ in turn by the minimizers of the augmented Lagrangian, followed by updating the dual variables $\lambda$ and $\phi$. This is performed iteratively until convergence. The details of each update are summarized as follows.

The update of $\mathbf m$ is obtained by solving a strictly convex quadratic program.  
Given $\boldsymbol{\rho}^t$, the first-order optimality condition leads to the sparse SPD system
\begin{align}\label{eq:update-m}
\small 
\big(W(\boldsymbol{\rho}^{t})+\beta\,\mathbf D^\top\mathbf D+\gamma I\big)\,\mathbf m^{t+1}
=\beta\,\mathbf D^\top\Delta\boldsymbol{\rho}^{t}-\mathbf D^\top\lambda^{t}
+\gamma\,\mathbf q^{t}-\phi^{t},    
\end{align}
which can be solved efficiently bythe  sparse Cholesky or PCG method.

The update of $\mathbf q$ amounts to an elementwise projection onto the FD hypograph evaluated at $\boldsymbol{\rho}^t$, i.e.,
\begin{align}\label{eq:update-q}
\mathbf q^{t+1}
=\Pi_{\mathcal K_{\rm FD}(\boldsymbol{\rho}^{t})}\!\Big(\mathbf m^{t+1}+\frac{\phi^{t}}{\gamma}\Big),   
\end{align}
which, under Greenshields, reduces to clamping between $0$ and $\mathcal Q_{e_\ell}(\bar\rho^{\,t}_{e_\ell,i})$.

The update of $\boldsymbol{\rho}$ is strictly convex on and inherits a block–tridiagonal structure from $\Delta$, so that
\begin{align}\label{eq:update-rho}
\begin{aligned}
\boldsymbol{\rho}^{t+1}
=\argmin_{\boldsymbol{\rho}}\;\;
\frac12 (\mathbf m^{t+1})^\top W(\boldsymbol{\rho})\,\mathbf m^{t+1}
-\lambda^{t\top}\Delta\boldsymbol{\rho}
+\frac{\beta}{2}\,\|\mathbf D\mathbf m^{t+1}-\Delta\boldsymbol{\rho}\|_2^2
\end{aligned}
\end{align}
that can be solved efficiently either by Newton iterations. Finally, the dual variables are updated by 
\begin{align}\label{eq:update-dual}
\begin{aligned}
\lambda^{t+1}&=\lambda^{t}+\beta\big(\mathbf D\mathbf m^{t+1}-\Delta\boldsymbol{\rho}^{t+1}\big),\\
\phi^{t+1}&=\phi^{t}+\gamma\big(\mathbf m^{t+1}-\mathbf q^{t+1}\big).    
\end{aligned}
\end{align}
The updating scheme can be thus summarized as follows.
\begin{algorithm}
\caption{ADMM iteration}
\begin{algorithmic}[1]
  \State Initialize $\boldsymbol{\rho}^0,\mathbf m^0,\mathbf q^0,\lambda^0,\phi^0$
  \Repeat
    \State update $\mathbf m^{t+1}$ by solving \eqref{eq:update-m}
    \State update $\mathbf q^{t+1}$ by FD projection \eqref{eq:update-q}
    \State update $\boldsymbol{\rho}^{t+1}$ by convex minimization \eqref{eq:update-rho}
    \State update $\lambda^{t+1},\phi^{t+1}$ by dual ascent \eqref{eq:update-dual}
  \Until{converged}
\end{algorithmic}
\end{algorithm}

\section{Numerical simulation}\label{sec:numerical}

\subsection{Single-lane line network}\label{subsec:line}
We first start with the line network, the simplest synthetic setting that captures the essentials of single-lane traffic (see, Figure \ref{fig:line}). We simulate on a 30-node directed line with parameters $\hat\rho=0.15$ and $v_0=3$. The time grid has eight snapshots $t=0/7,1/7,\dots,7/7$. Each node represents a spatial cell, and each directed edge a possible transfer of mass between adjacent cells. The model enforces conservation and capacity through the continuity equation and a fundamental-diagram constraint, and uses a convex objective that penalizes kinetic action. The midpoint scheme advances the state in time and yields a smooth density evolution from an upstream bump to a downstream target. Edge widths visualize momentum, and a centered-in-time averaging aligns the plot with the discretization, producing the expected symmetry. 

\begin{figure*}[htb!]
    \centering
    \includegraphics[width=1\linewidth]{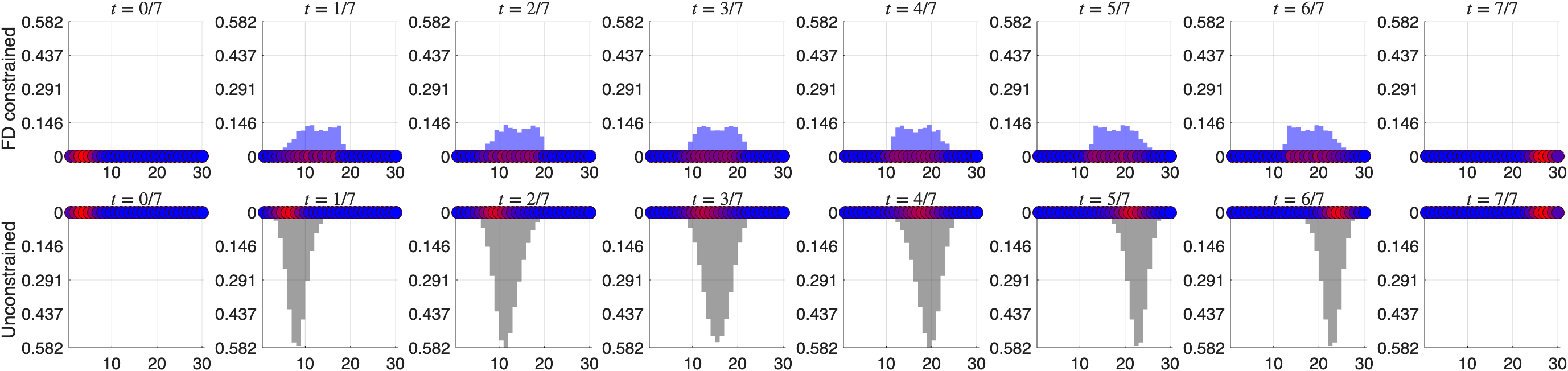}
    \caption{Evolution of densities on a 30-node line graph with $8$ snapshots at evenly spaced times $t=0/7,1/7,\dots,7/7$. Nodes are colored from blue (low density) to red (high density), with darker shades indicating higher values. Edges are drawn as semi-transparent blue bars, with thickness proportional to total momentum flow across each link, scaled consistently across snapshots. The dynamics are obtained from the midpoint discretization of the convex formulation, with symmetry enforced by averaging forward and backward momentum in time. $\hat \rho = 0.15$ and $v_0=3$.}
    \label{fig:line}
\end{figure*}

\subsection{Urban planar graph}
To complement the one-dimensional line graph experiment, we next consider a more realistic road topology obtained from the Athens metropolitan road network. Starting from publicly available GIS data \cite{yap2023global}, we generate a simplified planar subgraph with $n=291$ nodes and $m=614$ directed edges, preserving the sparse and locally planar structure typical of urban street layouts, see also \cite{solomon2016continuous}. Each node represents an intersection and carries the density variable, while directed edge pairs represent bidirectional road segments that support the momentum variables. This setting allows us to examine transport over a nontrivial spatial network with heterogeneous connectivity and multiple alternative routes. The problem is formulated as before, with density initialized near the upper-right portion of the network and transported toward target nodes near the bottom center. The Athens example in Figure \ref{fig:urban} highlights two key features: the role of network geometry in dispersing mass across multiple parallel routes, and the effect of capacity constraints (via the fundamental diagram) in shaping the evolution of traffic densities.

\begin{figure*}[htb!]
\centering
\includegraphics[width=1\linewidth]{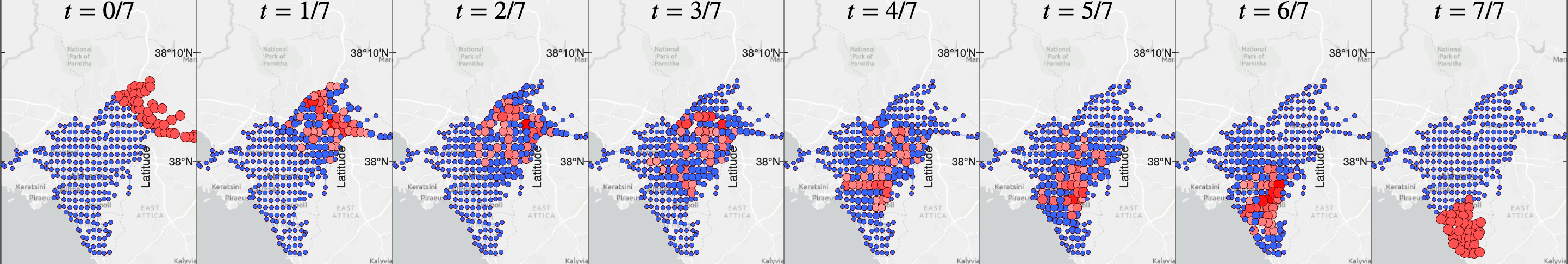}
\includegraphics[width=1\linewidth]{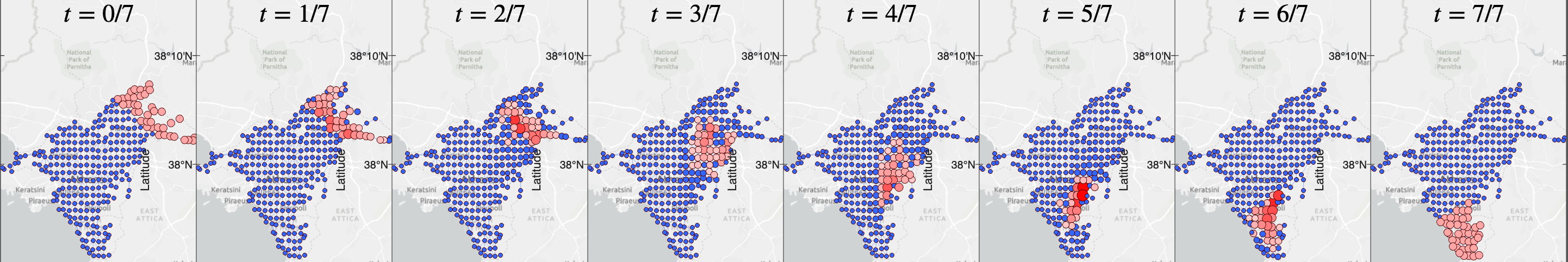}
\caption{Top row: congestion-aware transport on a directed network with $n=291$ nodes and $m=614$ directed edges. Bottom row: the same setting without the fundamental-diagram constraint. Snapshots are shown at $t\in\{0,\tfrac17,\tfrac27,\tfrac37,\tfrac47,\tfrac57,\tfrac67,1\}$ using a midpoint discretization with k=7 steps. For the top row, Greenshields capacities $\mathcal{Q}_e(\rho)=v_0\,\rho\!\left(1-\rho/\hat\rho\right)$ with $\hat\rho=0.05$ and $v_0=2$ are enforced on interior steps, and the endpoint frames are unconstrained. The bottom row uses the same solver and time grid but omits the capacity constraint. Node color and size encode density, with blue near zero and red at higher values and marker area proportional to mass.}
\label{fig:urban}
\end{figure*}

\subsection{Convergence results}\label{subsec:conv}

We first study the unconstrained case by setting the FD bound sufficiently large so that it is inactive throughout the evolution. We solve the six snapshots problem on a $30$ node directed line using the midpoint discretization and the ADMM scheme from Section~\ref{sec:optimization}.\footnote{Parameters: $\beta=60$, $\gamma=50$, midpoint discretization; stopping tolerance $10^{-8}$.} As a baseline, we compute the optimal objective with the MATLAB CVX solver~\cite{grant2008cvx}. Figure~\ref{fig:convergence_admm} presents the kinetic objective versus iteration (solid), together with the CVX optimum (dashed). The ADMM objective decreases smoothly and reaches the CVX value to plotting accuracy; the primal continuity and consensus residuals fall below $10^{-8}$ under the stated tolerances.

\begin{figure}[htb!]
    \centering
    \includegraphics[width=0.6\linewidth]{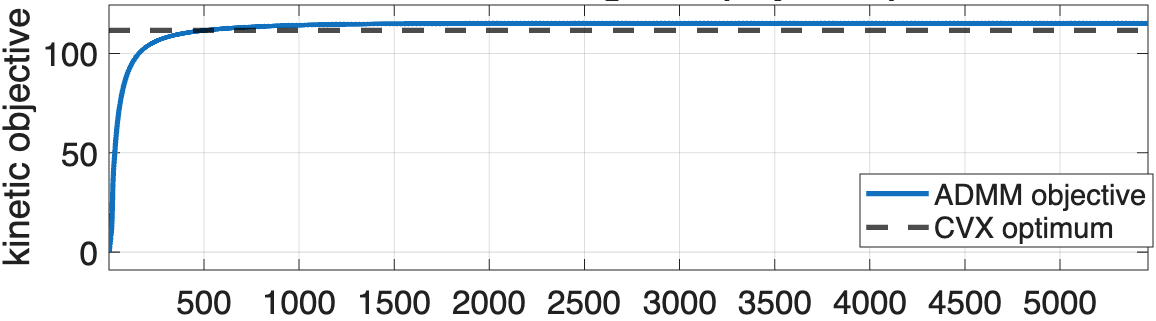}
    \caption{ADMM convergence on a line with six snapshots (FD inactive). The curve shows the kinetic objective versus iteration. The dashed line marks the CVX optimum.}
    \label{fig:convergence_admm}
\end{figure}

We next activate the capacity constraint using the Greenshields FD (with $v_0=1$ and $\hat\rho=0.10$), enforced on the interior time steps $i=2,\dots,k-1$ to match the CVX setup in Section~\ref{subsec:line}. The same ADMM parameters are used. Figure~\ref{fig:convergence_admm_fd} shows the resulting convergence: the objective approaches the CVX benchmark while the FD projection keeps all iterates feasible with respect to the per–edge capacities.

\begin{figure}[htb!]
    \centering
    \includegraphics[width=0.6\linewidth]{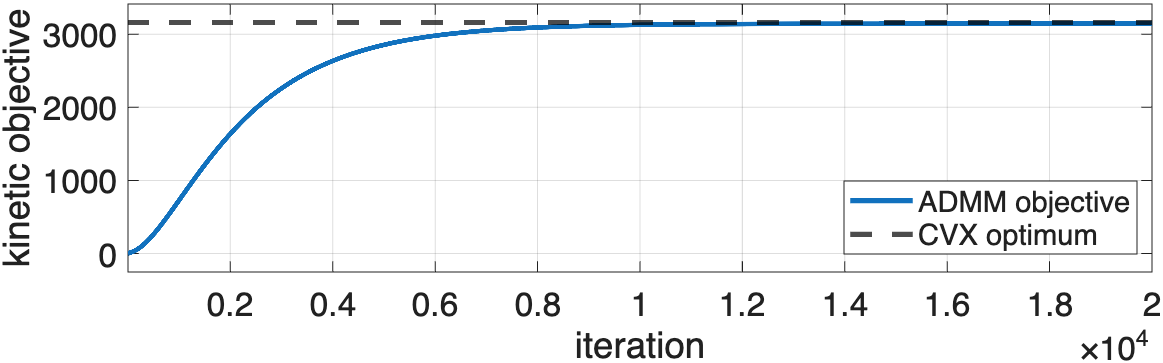}
    \caption{ADMM convergence with Greenshields FD active ($v_0=1$, $\hat\rho=0.10$). The projection step enforces the FD constraint at every iteration; the objective approaches the CVX benchmark under the same ADMM parameters as in Figure~\ref{fig:convergence_admm}.}
    \label{fig:convergence_admm_fd}
\end{figure}

\section{Conclusion}\label{sec:conclusion}
We formulate dynamic optimal transport on directed graphs with edgewise capacity limits dictated by a fundamental diagram, derive a midpoint discretization that yields a convex program with a unique minimizer, and design an augmented-Lagrangian splitting method with simple, scalable subproblems. Case studies on a line and a city-scale planar subgraph illustrate congestion-aware displacement and convergence to convex baselines, indicating suitability for planning and closed-loop scheduling on capacity-limited networks.

\bibliographystyle{plain}
\bibliography{references}

\end{document}